\documentclass[12pt,a4paper]{article}

\usepackage[latin1]{inputenc}
\usepackage[english]{babel}
\usepackage[usenames, dvipsnames]{xcolor}
\usepackage{amssymb}
\usepackage{amsfonts}
\usepackage{amsmath}
\usepackage{amsthm}
\usepackage{epsfig}
\usepackage{graphics}
\usepackage{dsfont}

\usepackage{verbatim}

\usepackage{latexsym}

\usepackage{amsthm}
\usepackage{cleveref}
\usepackage{mathtools}
\usepackage{bbm}
\usepackage[shortlabels]{enumitem}

\newcommand{\brc}[1]{\left(#1\right)}

\newcommand{\folge}[3][\N]{\left({#2}_{#3}\right)_{{#3}\in {#1}}}
\newcommand{\subfolge}[4][\N]{\left({#2}_{{#3}_{#4}}\right)_{{#4}\in {#1}}}
\newcommand{\subsubfolge}[5][\N]{\left({#2}_{{#3}_{{#4}_{#5}}}\right)_{{#5}\in {#1}}}
\newcommand{\norm}[1]{\left\Vert #1 \right\Vert}

\newcommand{\QED}{\begin{flushright}$\qed$\end{flushright}}

\newcommand{\mc}[1]{\mathcal{#1}}

\newcommand{\toinf}{\rightarrow\infty}

\newcommand{\xr}[2][]{\xlongrightarrow[#1]{#2}}
\newcommand{\xR}[2][]{\xRightarrow[#1]{#2}}
\newcommand{\xRa}{\xR{alm}}

\newcommand{\inv}{^\leftarrow}

\newcommand{\intr}{\int\limits_\R}
\newcommand{\intl}[1][\R]{\int\limits_{#1}}

\newcommand {\Q}	{\mathbb{Q}}
\newcommand {\R}	{\mathbb{R}}
\newcommand {\N}	{\mathbb{N}}

\newcommand{\Ra}	{\Rightarrow}

\newcommand{\ra}{\rightarrow}

\newcommand{\abs}[1]{\left| #1 \right|}

\renewcommand{\inv}{^{-1}}

\newcommand{\supp}{\text{supp}}

\renewcommand{\emph}[1]{\textit{#1}}
\newcommand{\prooof}{\textit{Proof. }}

\newcommand{\hoch}[1]{^{\left(#1\right)}}
\newcommand{\V}{\mathcal{V}}
\newcommand{\W}{\mathcal{W}}

\newcommand{\ind}{\mathbbm{1}}

\long\def\/*#1*/{}
 
 \newcommand{\ntoinf}[1][n]{,\quad #1\toinf}

\newcommand{\dis}[1]{F\hoch{#1}}

 \newcommand{\vag}{{loose }}
 \newcommand{\vagg}{{loose}}
 \newcommand{\vagly}{{loosely }}
 \newcommand{\cv}{\xr{l}}

\theoremstyle{definition}
\newtheorem{definition}{Definition}[section]
\crefname{definition}{Definition}{Definitions}
\newtheorem{example}[definition]{Example}
\crefname{example}{Example}{Examples}

\theoremstyle{plain}
\newtheorem{theorem}[definition]{Theorem}
\crefname{theorem}{Theorem}{Theorems}
\newtheorem{corollary}[definition]{Corollary} 
\crefname{corollary}{Corollary}{Corollaries}
\newtheorem{lemma}[definition]{Lemma}
\crefname{lemma}{Lemma}{Lemmata}

\crefname{proposition}{Proposition}{Propositions}

\theoremstyle{remark}
\newtheorem{remark}[definition]{Remark}
\crefname{remark}{Remark}{Remarks}

\crefname{question}{Question}{Questions}

\begin{document}
\renewcommand{\bibname}{References}
\begin{center}
    
\begin{huge}
    Vague and basic convergence of signed measures
\end{huge}
\vspace{0.2cm}
\\Michael Stan\v ek\footnote[1]{Institute of Mathematical Finance, Ulm University, Ulm, Germany\\
michael.stanek@uni-ulm.de}\\
\vspace{0.2cm}
\vspace{0.5cm}
\end{center}
\begin{center}
    \textbf{Abstract}\\
\end{center}
We study the relationship between different kinds of convergence of finite signed measures and discuss their metrizability. In particular, we study the concept of basic convergence recently introduced by Khartov \cite{khartov2022weak} and introduce the related concept of almost basic convergence. We discover that a sequence of finite signed measures converges vaguely if and only if it is locally uniformly bounded in variation and the corresponding sequence of distribution functions either converges in Lebesgue measure up to constants, converges basically, or converges almost basically.

\section{Introduction}
\label{c:intro}

Consider finite signed Borel measures $\mu_n,n\in\N,$ and $\mu$ on $\R$. Then the sequence $\folge{\mu}{n}$ is commonly said to converge \emph{weakly} to $\mu$, if
\begin{equation}
\label{eq:intro}
\intr fd\mu_n\ra\intr fd\mu    
\end{equation}
holds for all continuous, bounded $f:\R\ra\R$.\\
A weaker form of this convergence is \emph{vague} convergence. Here, we say that $\mu_n$ converges to $\mu$ vaguely, if \cref{eq:intro} holds for all continuous $f$ with compact support. Some sources define vague convergence using continuous functions vanishing at infinity instead. In this paper, we call this convergence \emph{\vag}convergence instead to distinguish it from vague convergence. \\

Another possibility to define convergence of signed measures is through convergence (in some sense) of their distribution functions, defined by $\dis{\mu_n}(x)\coloneqq\mu_n((-\infty,x])$. Here, we will consider two kinds of such defined convergences: basic convergence and almost basic convergence. Basic convergence has been defined by Khartov in \cite{khartov2022weak} to study convergence of quasi-infinitely divisible distributions. Following his definition, the sequence $\folge{\mu}{n}$ is said to converge basically to $\mu$ if each subsequence $\subfolge{\mu}{n}{k}$ contains a further subsequence $\subsubfolge{\mu}{n}{k}{l}$ such that $\dis{\mu_{n_{k_l}}}(x)-\dis{\mu_{n_{k_l}}}(y)$ converges to $\dis\mu(x)-\dis\mu(y)$ for all $x,y\in\R\setminus S$, where $S$ is an at most countable set and may depend on the subsequence. Note that this is equivalent to $\mu_{n_{k_l}}((x,y])$ converging to $\mu((x,y])$ for all except for countable many $x,y\in\R$ with $x\leq y$. \\
This definition can be slightly relaxed by demanding the pointwise convergence of differences to hold not for all except at most countable many $x,y$, but only for almost all $x,y$ with respect to the Lebesgue measure. We call this relaxation \emph{almost basic} convergence.\\
Clearly, weak convergence implies \vag convergence, which itself implies vague convergence, and basic convergence implies almost basic convergence. As we will see, there are also relationships between vague convergence and (almost) basic convergence and between \vag convergence and (almost) basic convergence.\\
In Chapter 2, we define vague, {\vagg}, basic and almost basic convergence and give some examples of basically convergent sequences, which illustrate why details in the definition cannot be omitted. We will then further study the relationship between these four kinds of convergence and the question of their metrizability and find a characterization of basic and almost basic convergence in Chapter 3. As we shall see, almost basic convergence is metrizable, while vague, \vag and basic convergence are not. The main result is \cref{theo:micha_helly}, where we will see that a sequence of finite signed measures is vaguely convergent if and only if it is locally uniformly bounded in variation and either basically or almost basically convergent.\\
These results can be seen as a more general version of the classical Helly-Bray Theorem for finite non-negative measures, giving not only sufficient, but also necessary conditions for vague convergence and considering signed measures. They generalize results from Herdegen et al. \cite{herdegen2022vague}, where the relationship between vague convergence of signed measures and convergence of their distribution functions has been studied as well. Since there pointwise convergence of distribution functions was considered, instead of basic convergence, conditions which are both sufficient and necessary could only be found in the case where the sequence of distribution functions satisfies the condition of having \emph{no mass} for all $x
\in\R
$. Even though this condition seems to be quite natural, we will see in \cref{ex:mass_everywhere} that there exist sequences of distributions which do not satisfy it, but still converge vaguely - hence a more general statement which can also be applied to those sequences is of interest.\\
Moreover, \cref{theo:micha_helly} puts basic convergence into relation to vague convergence and hence gives further insight into convergence properties of quasi-infinitely distributions. Specifically, we see that since local boundedness in variation at least of the negative parts was assumed in all theorems in Khartov \cite{khartov2022weak} which involve basic convergence, basic convergence can there be replaced by vague convergence. When boundedness in variation is assumed instead of local boundedness in variation or boundedness in variation only of the negative or positive parts, basic convergence can even be replaced by the stronger condition of \vag convergence.\\
\cref{theo:micha_helly} implies in particular that a sequence of finite signed measures $\folge{\mu}{n}$ converges vaguely to some finite signed measure $\mu$ if and only if $\folge{\mu}{n}$ is locally uniformly bounded in variation and there exists a sequence $\folge{c}{n}$ of real numbers such that $\dis{\mu_n}+c_n$ converges to $\dis\mu$ in Lebesgue measure.

\section{Definitions and Examples}
\label{c:conv}

\subsection*{Notation}
We denote $\N=\{1,2,\ldots\}$ and $\N_0=\{0,1,2,\ldots\}$. With (signed) measures, we always mean (signed) Borel measures on $\R$. We denote the Lebesgue measure by $\lambda$ and, for $x\in\R$, the Dirac measure at $x$ by $\delta_x$. For a finite signed measure $\mu$, we denote its positive and negative part according to the Hahn-Jordan decomposition by $\mu^+$ and $\mu^-$ and by $\abs{\mu}\coloneqq\mu^++\mu^-$ its total variation. We denote the set of real-valued bounded continuous functions on $\R$ by $C_b(\R)$, the set of real-valued continuous functions with compact support on $\R$ by $C_c(\R)$ and the set of real-valued continuous functions on $\R$ which vanish at infinity by $C_0(\R)$. By $C^1(\R)$ we denote functions which are at least once continuously differentiable, and write $C_c^1(\R)=C_c(\R)\cap C^1(\R)$. We denote the supremum norm for functions by $\norm{\cdot}_\infty$. For any set $A$, we denote by $\ind_A$ the function for which $f(x)=1$ if $x\in A$ and $f(x)=0$ otherwise. By $\V$ we denote the set of right-continuous functions of bounded variation mapping from $\R$ to $\R$, by $\W$ the subspace of all $f\in\V$ with $\lim_{x\ra-\infty}f(x)=0$. Here, bounded variation means that $\lim_{n\toinf}TV_{[-n,n]}f<\infty$, where $TV_{[a,b]}$ denotes the total variation of $f$ on $[a,b]$. For a finite signed measure $\mu$, we define its distribution function $\dis\mu:\R\ra\R$ by
$$\dis\mu(x)=\mu((-\infty,x]).$$
Then $\dis\mu\in\W$. Moreover, it is well known that we  have a one-to-one correspondence between finite signed measures and functions in $\W$.

\subsection*{Definitions of convergence types}

Commonly, convergence of a sequence of finite signed measures $\folge{\mu}{n}$ to a finite signed measure $\mu$ is defined by convergence
\begin{equation}
    \label{eq:int_conv}
    \lim_{n\toinf}\intr fd\mu_n=\intr fd\mu
\end{equation}
for all functions $f$ in a suitable class of functions.\\ 
\begin{definition}   
We say, that $\mu_n$ converges to $\mu$
\begin{enumerate}
    \item \emph{vaguely}, and write $\mu_n\xr{v}\mu$, if \cref{eq:int_conv} holds for all $f\in C_c(\R)$,
    \item \emph{\vagly}, and write $\mu_n\cv\mu$, if \cref{eq:int_conv} holds for all $f\in C_0(\R)$,
    \item \emph{weakly}, and write $\mu_n\xr{w}\mu$, if \cref{eq:int_conv} holds for all $f\in C_b(\R)$.
\end{enumerate}
\end{definition}

Here, we follow the most common definition of vague convergence in the literature, which is used for example in Herdegen et al. \cite{herdegen2022vague}, Bauer \cite[30.1]{bauer1992mass}, Berg et al. \cite[Chapter 2, \S4]{berg1984harmonic}, Dieudonn\'e \cite[7, Section XIII.4]{dieudonne1970treatise},  Kallenberg \cite[Chapter 5, p. 98]{kallenberg1997foundations} or Klenke \cite[13.12]{klenke2008probability}. Sometimes vague convergence is defined using functions which vanish at infinity instead of functions with compact support. While these definitions coincide for the class of probability measures, they are in general not equivalent for signed measures.\\
Note, that weak, vague and \vag convergence can be interpreted as weak-* convergence with respect to $C_b, C_c$ respectively $C_0$ by identifying a function $f$ with the functional $\mu\mapsto\intr fd\mu$.\\

Alternatively, one can consider convergence (in some way) of the distribution functions of $\mu_n, \mu$ instead. One such kind convergence has been defined by A. Khartov in \cite{khartov2022weak} to study convergence of quasi-infinitely divisible distributions.
\begin{definition}\cite{khartov2022weak}
\label{def:bas_conv}
    Let $\folge{f}{n},f$ be functions mapping from $\R$ to $\R$. We say that $\folge{f}{n}$  \emph{converges basically} to $f$ and write $f_n\Ra f$, if every subsequence $\subfolge{f}{n}{k}$ contains a further subsequence $\subsubfolge{f}{n}{k}{l}$ such that
    \begin{equation}
    \label{eq:conv}
    f_{n_{k_l}}(x_1)-f_{n_{k_l}}(x_2)\ra f(x_1)-f(x_2),\quad l\toinf     
    \end{equation}
    for all $x_1, x_2\in\R\setminus S$, where $S$ is an at most countable set and may depend on the subsequence.\\
    For signed measures $\mu$ and $\mu_n,n\in\N$, we say that $\mu_n$ \emph{converges basically} to $\mu$ for $n\toinf$ and write $\mu_n\Ra\mu\ntoinf$, if the corresponding distribution functions converge basically, i.e. if $\dis{\mu_n}\Ra\dis\mu$ for $n \toinf$, i.e. if for each subsequence there exists a further subsequence such that $\mu_{n_{k_l}}((x_1,x_2])\ra\mu((x_1,x_2])$ for $n \toinf$ for all $x_1,x_2\in\R\setminus S$ with $x_1\leq x_2$.
\end{definition}

Finally, we will also work with a slight relaxation of basic convergence, where we only demand \cref{eq:conv} to hold almost everywhere instead of everywhere except for an at most countable set.

\begin{definition}
\label{def:almost_bas_conv}
    Let $\folge{f}{n},f$ be functions mapping from $\R$ to $\R$. We say that $\folge{f}{n}$  \emph{converges almost basically} to $f$ and write $f_n\xRa f$, if every subsequence $\subfolge{f}{n}{k}$ contains a further subsequence $\subsubfolge{f}{n}{k}{l}$ such that \cref{eq:conv} holds for all $x_1, x_2\in\R\setminus S$ where $S$ is a set of Lebesgue measure zero that may depend on the subsequence.\\
    For signed measures $\mu$ and $\mu_n,n\in\N$, we say that $\mu_n$ \emph{converges almost basically} to $\mu$ for $n\toinf$ and write $\mu_n\xRa\mu\ntoinf$, if the corresponding distribution functions converge almost basically, i.e. $\dis{\mu_n}\xRa\dis\mu$ for $n\toinf$, i.e. if for each subsequence there exists a further subsequence such that $\mu_{n_{k_l}}((x_1,x_2])\ra\mu((x_1,x_2])$ for $n\toinf$ for $\lambda$-almost all $x_1,x_2\in\R$ with $x_1\leq x_2$.
\end{definition}

\begin{remark}
\label{rem:bas_lim_unique_rem}
    \begin{enumerate}[label=(\alph*)]
        \item Basic convergence clearly implies almost basic convergence. Conversely, using indicator functions over sets of the form $$A_n=\bigcap_{m=1}^n\bigcup_{k=0}^{3^{m-1}-1}\bigg[\frac{3k}{3^m},\frac{3k+1}{3^m}\bigg)\cup \bigg[\frac{3k+2}{3^m},\frac{3k+3}{3^m}\bigg),$$ whose intersection differs from the Cantor set just by countable many points, we can construct a sequence of right-continuous functions of bounded variation which converges to $1$ on an uncountable set of measure zero and to $0$ everywhelse else. This sequence then converges to zero almost basically, but not basically. Hence almost basic convergence is in general indeed a weaker condition than basic convergence.
        \item \label{rem:bas_lim_unique} Clearly, the basic, respectively almost basic limit is unique almost everywhere, up to a constant. Since in $\V$ and $\W$ we work with right-continuous functions, the limit in $\W$ is unique, while the limit in $\V$ is unique up to a constant.
    \end{enumerate}
\end{remark}

\subsection*{Examples}
In the following, we give a few examples of basically convergent sequences that illustrate that details in the definition cannot be omitted.
The first example illustrates that it is indeed necessary to consider the differences $f_n(x)-f_n(y)$ instead of $f_n(x)$:
\begin{example}
\label{ex:bas_unb}
    Let $\mu_n=n\cdot\delta_{-n}$ for all $n\in\N$ and $\mu\equiv 0$. Then $\dis{\mu_n}(x)=
        0$ if $x<-n$ and $\dis{\mu_n}(x)=n$ elsewise. We see that $\dis{\mu_n}(x)$ does not converge for any $x\in\R$. However, for all $x,y\in\R$ and all $n\geq\max(\abs{x},\abs{y})$ we have $\dis{\mu_n}(x)-\dis{\mu_n}(y)=0$, hence $\mu_n\Ra\mu$. Note that also $\mu_n\xr{v}\mu$, since for all $r>0$ we have $\abs{\mu_n}([-r,r])=0$ for all $n\geq r$. Crucially, $\folge{\mu}{n}$ is uniformly bounded in variation on bounded intervals in this case, even though it is not bounded in variation on $\R$.
\end{example}

In general, it is necessary to consider subsequences, and the at most countable set $S$, for which the convergence in \cref{eq:conv} does not hold, may depend on the choice of the particular subsequence. This becomes obvious in the following example, the sometimes so-called \emph{typewriter sequence}, as found, inter alia, in Khartov \cite{khartov2022weak} and Bogachev \cite[Example 1.4.8]{bogachev2018weak}:

\begin{example}
\label{ex:no_mass_vague}
For all $n\in\N$, let $k_n\in\N_0$ be such that $2^{k_n}\leq n<2^{k_n+1}$. Define the sequence $\folge{\mu}{n}$ of signed measures  by
$$\mu_n=\brc{\delta_{a_n}-\delta_{b_n}},\quad \text{ where } a_n={\frac{n-2^{k_n}}{2^{k_n}}}\text{ and } b_n=\frac{n+1-2^{k_n}}{2^{k_n}}.$$
Then $\dis{\mu_n}=\ind_{[a_n,b_n)}$. The interval $[a_n,b_n)$ is shifting over $[0,1)$ and vanishes for $n\toinf$, since $b_n-a_n=\frac{1}{2^{k_n}}\ra0$ as $n\toinf$. Hence, for any $x,y\in[0,1)$ with $x\neq y$, for infinitely many $n\in\N$ we have $\dis{\mu_n}(x)=\dis{\mu_n}(y)=0$, but also for infinitely many $n\in\N$ we have $\dis{\mu_n}(x)=0$ and $\dis{\mu_n}(y)=1$. Therefore, $\dis{\mu_n}(x)-\dis{\mu_n}(y)$ does not converge unless we take a subsequence.\\
Let $\mu\equiv 0$. Since functions in $C_b(\R)$ are uniformly continuous on $[0,1]$ and $b_n-a_n\ra0$ as $n\toinf$,
$$\intr fd\mu_n\ra\intr fd\mu\ntoinf$$
for all $f\in C_b(\R)$, i.e. $\mu_n\xr{w}\mu$ for $n\toinf$. This implies $\mu_n\Ra\mu$ for $n\toinf$ by Bogachev \cite[Thm. 1.4.7]{bogachev2018weak}, which we will also state below as \cref{theo:bogachev_basic} for the convenience of the reader.
\end{example}

The following example illustrates that basic convergence on its own does not imply vague convergence:
\begin{example}
\label{ex:bas_unb_komp}
    Let $\mu_n=n^2\cdot\delta_0-n^2\cdot\delta_{\frac{1}{n}}$ for all $n\in\N$ and $\mu\equiv 0$ be signed measures. Then for all $n\in\N$ we have $\dis{\mu_n}=n^2\cdot\ind_{[0,\frac{^1}{n})}$. Thus we see that for all $x\in\R\setminus \{0\}$ and all $n\geq \frac{1}{\abs{x}}$ we have $\dis{\mu_n}(x)=0$, hence $\mu_n\Ra\mu$.\\
    Note that this sequence does not converge vaguely: 
    Let $f\in C_c(\R)$ be such that $f(x)=x$ for all $x\in[0,1]$. Then
    $$\intr fd\mu_n=n$$
    for all $n\in\N$, which obviously does not converge.
\end{example}

\section{Main results}
\label{c:main}
\subsection*{Topological results}
It is well known that weak convergence of sequences is topological and arises from a non-metrizable topology, see e.g. Bogachev \cite[3.1]{bogachev2018weak}. This, however, does not directly imply non-metrizability of weak convergence, since there could exist another topology, possibly metrizable, inducing the same (sequential) convergence. We expect the following result to be known, but could not find a ready reference and hence provide a short argument, also because in \cref{cor:vague_nonmet} the same argument will be applied for vague and \vag convergence.
\begin{theorem}
\label{theo:weak_notmet}
    Weak convergence is not metrizable on the space of all finite signed measures on $\R$.
\end{theorem}

\begin{proof}
    Assume weak convergence was metrizable. Let $d$ be a metric such that $d(\mu_n,\mu)\ra 0$ if and only if $\mu_n\xr{w}\mu$. For all $m\in\N$, we have $m(\delta_0-\delta_{\frac{1}{n}})\xr{w}0$ as $n\toinf$. That means that for all $m\in\N$ there exists an $N_m>0$ such that $d(m(\delta_0-\delta_{\frac{1}{n}}),0)\leq\frac{1}{m}$ for all $n\geq N_m
$. Consider the sequence $\folge{\mu}{m}$ defined by $\mu_m\coloneqq m(\delta_0-\delta_{\frac{1}{N_m}})$. Then $d(\mu_m,0)\leq\frac{1}{m}$ for all $m\in\N$. However, $\mu_m$ cannot converge weakly to $0$ for $m\toinf$, as it is not uniformly bounded in variation, see Bogachev \cite[Thm. 1.4.7]{bogachev2018weak}, stated in \cref{theo:bogachev_basic}.
\end{proof}

One big difference between basic and almost basic convergence is that almost basic convergence is metrizable, while basic convergence is not.

\begin{lemma}
\label{lemma:basic_not_met}
  Basic convergence is not metrizable on $\W$.\\
\end{lemma}
\prooof
Assume basic convergence was metrizable on $\W$. Let $d$ be a metric on $\W$ such that $d(f_n,f)\ra0$, $n\toinf$, if and only if $f_n\Ra f$.
  \/*It follows directly from the definition of basic convergence that addition and scalar multiplication are continuous operations with relation to basic convergence, hence $\V$ with basic convergence is a metrizable topological vector space. This implies that there exists a translation-invariant metric which induces basic convergence, see \cite[\S15.11]{kothe1969topological}. Hence we can assume without loss of generality that $d$ is translation-invariant, i.e. for all $f,g,h\in\V$ we have $d(f,g)=d(f+h,g+h)$. Then $d(f+g,0)\leq d(f+g,g)+d(g,0)=d(f,0)+d(g,0)$ for all $f,g\in\V$.*/
  
  We will construct a sequence of functions such that $d(f_n,0)\ra0$ as $n\toinf$, but $f_n\not\Ra 0$ using a Cantor-like set.\\
  Let $I_0=[0,1)$ and $f_0=\ind_{I_0}$. Since $\ind_{\big[0,\frac{1}{n}\big)\cup\big[1-\frac{1}{n},1\big)}\in\W$ converges to zero pointwise on $\R\setminus\{0\}$ for $n\toinf$, it also converges basically to $0$, hence there exist $a, b\in(0,1)$ with $a<b$ such that $d(\ind_{[0,a)\cup[b,1)},0)<\frac{1}{2}$. Let $f_1=\ind_{[0,a)\cup[b,1)}$. Now, the functions $\ind_{\big[0,\frac{1}{n}\big)\cup\big[a-\frac{1}{n},a\big)\cup \big[b,b+\frac{1}{n}\big)\cup\big[1-\frac{1}{n},1\big)}\in\W$ converge basically to zero, hence there exist $a_1,a_2,b_1,b_2$ with $0<a_1<a_2<a<b<b_1<b_2<1$ such that $d(\ind_{[0,a_1)\cup[a_2,a)\cup[b,b_1)\cup[b_2,1)},0)\leq\frac{1}{3}$.
  Repeating this construction, we can inductively construct a sequence of functions $\folge{f}{n}$ and sets of disjoint half-open intervals $\{I_{n,1},\ldots,I_{n,2^n}\},\quad n\in\N$, such that for all $n\in\N$:
  \begin{itemize}
    \item $I_{n,2k-1}$ and $I_{n,2k}$ are disjoint subsets of $I_{n-1,k}$ for all $k\in\{1,\ldots,2^{n-1}\}$,
    \item $f_n=\ind_{\cup_{k=1}^{2^n}I_{n,k}}$,
    \item $d(f_n,0)<\frac{1}{n}$ and
    \item $\lim_{n\toinf}\sup_{k=1,\ldots,2^{n}}\lambda(I_{n,k})=0$
  \end{itemize}
  Then $d(f_n,0)\ra 0$ as $n\toinf$. However, the pointwise limit of $f_n$ is the function $\ind_{C}$, where $C=\bigcap_{n=1}^\infty\bigcup_{k=1}^{2^n}I_{n,k}$. $C$ is very similar to the classical Cantor set, and in particular is uncountable. This can be seen from the fact that $C$ and $\bigcap_{n=1}^\infty\bigcup_{k=1}^{2^n}\overline{I_{n,k}}$ differ only by a countable set, where $\overline{I_{n,k}}$ denotes the closure of $I_{n,k}$. 
  The set $\bigcap_{n=1}^\infty\bigcup_{k=1}^{2^n}\overline{I_{n,k}}$ is compact, non-empty and has no isolated points (i.e. is a perfect set), hence it is uncountable by Hewitt and Stromberg \cite[Theorem 6.65]{hewitt1965real}.
  Since $f_n$ converges pointwise to a function which is different from zero on an uncountable set, but equals to zero on a set of positive measure, so does each of its subsequences, hence $f_n$ cannot converge basically to zero.
  Thus basic convergence is not metrizable.\QED

\begin{remark}
    Clearly, basic convergence is also not metrizable on $\V$, since the basic limit is not unique in $\V$. This problem could be solved by considering the quotient space of $\V$ over suitable equivalence classes. However, the resulting space would then be isomorphic to $\W$, leading us back to \cref{lemma:basic_not_met}.
\end{remark}

In contrast, almost basic convergence is metrizable on $\W$. The following lemma provides an explicit metric, in a fashion similar to the Ky Fan metric $\alpha(X,Y)=\inf\{\varepsilon>0:P(\abs{X-Y}>\varepsilon\}\leq\varepsilon\}$ for random variables $X$ and $Y$, which metrizes convergence in probability, see e.g. Dudley \cite[Thm. 9.2.2]{dudley2002real}.
\begin{lemma}
\label{lemma:micha_metric}
  Almost basic convergence is metrizable on $\W$. A metric is given by 
 $$d(f,g)=\min\left\{\varepsilon\geq0:\exists c\in\R:\lambda\left(\left\{x\in \left[-\frac{1}{\varepsilon},\frac{1}{\varepsilon}\right] : \abs{f(x)-c-g(x)}>\varepsilon\right\}\right)\leq\varepsilon\right\},$$
where we use the convention $\left[-\frac{1}{0},\frac{1}{0}\right]\coloneqq\R$.\\  
\end{lemma}
\prooof 
First of all, we show that the minimum exists.
  For this, replace the minimum in the definition of $d$ by an infimum, so that $d(f,g)=\inf E$, where 
  $$E\coloneqq\left\{\varepsilon\geq 0:\exists c\in\R:\lambda\left(\left\{x\in\left[-\frac{1}{\varepsilon},\frac{1}{\varepsilon}\right]:\abs{f(x)-c-g(x)}>\varepsilon\right\}\right)\leq\varepsilon\right\}.$$
  It is clear that  $E$ is an interval of the form $(b,\infty)$ or $[b,\infty)$ for some $b\in[0,\sqrt{2})$.
  Consider first the case that $d(f,g)>0$. Let $\folge{\varepsilon}{n}$ be a sequence of real numbers such that $\varepsilon_n\ra d(f,g)$ as $n\toinf$ and $d(f,g)<\varepsilon_n<\sqrt{2}$ and
  \begin{equation}
  \label{eq:nonmet_1}
      \frac{1}{d(f,g)}-\frac{1}{\varepsilon_n}<\frac{1}{2n}
  \end{equation} for all $n\in\N$. Let $\folge{c}{n}\subset\R$ and let $\folge{S}{n}$ be a sequence of Borel sets such that $\lambda(S_n)\leq\varepsilon_n$ and $\abs{f(x)-c_n-g(x)}\leq\varepsilon_n$ for all $x\in [-1/\varepsilon_n,1/\varepsilon_n]\setminus S_n$ and $n\in\N$.\\
  Since $f,g\in\W$, the function $f-g$ is bounded, hence $\folge{c}{n}$ must be bounded as well and hence contains a subsequence $\subfolge{c}{n}{k}$ converging to some $c\in\R$.
   Now define
   $$\tilde S_n\coloneqq[-1/d(f,g),-1/\varepsilon_n]\cup S_n \cup [1/\varepsilon_n,1/d(f,g)]$$ for all $n\in\N$.
   Then $\lambda(\tilde S_n)\leq\varepsilon_n+1/n$ due to \cref{eq:nonmet_1} and $\abs{f(x)-c_n-g(x)}\leq\varepsilon_n$ for all $x\in[-1/d(f,g),1/d(f,g)]\setminus \tilde{S}_n$ and all $n\in\N$.\\
   Now, let $\delta>0$ be arbitrary, but fixed. Then there exists $K_0\in\N$ such that for all $k>K_0$ we have $\abs{c_{n_k}-c}\leq\delta/2$ and $\varepsilon_{n_k}\leq d(f,g)+\delta/2$. Then due to the triangle inequality $\abs{f(x)-c-g(x)}\leq d(f,g)+\delta$ for all $k>K_0$ and $x\in [-1/d(f,g), 1/d(f,g)]\setminus \tilde{S}_{n_k}$. We define $S\coloneqq\bigcup_{k=1}^\infty\bigcap_{j=k}^\infty \tilde{S}_{n_j}$. Then $\abs{f(x)-c-g(x)}\leq d(f,g)+\delta$ for all $x\in[-1/d(f,g),1/d(f,g)]\setminus S$. Since $\delta$ was arbitrary, it follows that actually $\abs{f(x)-c-g(x)}\leq d(f,g)$ for all $x\in [-d(f,g),d(f,g)]\setminus S$. Moreover, by the continuity of measures $\lambda\brc{S}\leq d(f,g)$. Hence, the minimum is attained if $d(f,g)>0$.\\
   Now let $d(f,g)=0$. Then for all $n\in\N$ there exist a set $S_n$ of measure at most $1/n$ and a constant $c_n$ such that $\abs{f-c_n-g}\leq 1/n$ on $[-n,n]\setminus S_n$. Like above, there exists a subsequence $\subfolge{c}{n}{k}$ converging to some $c\in\R$. Let $K\in\N$ be arbitrary. Then by the triangle inequality $\abs{f(x)-c-g(x)}\leq 1/n+\abs{c_n-c}$ for all $x\in[-K,K]\setminus S_n$ and $n\geq K$. 
   For every $\gamma>0$ we then see 
   $$\lambda\left(\left\{x\in\left[-K,K\right]:\abs{f(x)-c-g(x)}>\gamma\right\}\right)\leq\gamma$$ by choosing $n$ such that $1/n+\abs{c_n-c}<\gamma$. This implies $f=g+c$ almost everywhere in $[-K,K]$ and hence almost everywhere in $\R$. Hence the minimum is also attained if $d(f,g)=0$.

   Clearly, $d$ is nonnegative and symmetric, and $d(f,g)=0$ if and only if $f=g+c$ almost everywhere for some $c\in\R$. Since $\lim_{x\ra-\infty}f(x)=0=\lim_{x\ra-\infty}g(x)$ by definition of $\W$, $c=0$, and as $f,g$ are continuous from the right, almost everywhere equality implies equality. Hence $d(f,g)=0$ if and only if $f=g$.\\
   Let $f,g,h\in\W$ all be different. 
   Let $d_1\coloneqq d(f,h)$ and $d_2\coloneqq d(h,g)$. Then there exist constants $c_1,c_2\in\R$ and sets $S_1,S_2\in\mc{B}(\R)$ such that $\lambda(S_i)\leq d_i$ for $i\in\{1,2\}$ and $\abs{f(x)-c_1-h(x)}\leq d_1$ for all $x\in[-1/d_1,1/d_1]\setminus S_1$ and $\abs{h(x)-c_2-g(x)}\leq d_2$ for all $x\in[-1/d_2,1/d_2]\setminus S_2$. But then, since $[-1/(d_1+d_2),1/(d_1+d_2)]\setminus (S_1\cup S_2)\subset([-1/d_1,1/d_1]\setminus S_1)\cap([-1/d_2,1/d_2]\setminus S_2)$ we have $\abs{f(x)-(c_1+c_2)-g(x)}\leq\abs{f(x)-c_1-h(x)}+\abs{h(x)-c_2-g(x)}\leq d_1+d_2$ for all $x\in[-1/(d_1+d_2),1/(d_1+d_2)]\setminus (S_1\cup S_2)$. Since $\lambda(S_1\cup S_2)\leq d_1+d_2$, we see that $d(f,g)\leq d(f,h)+d(h,g)$. Hence $d$ is indeed a metric.\\

   Let $f_n,f\in\W$ with $d(f_n,f)\ra 0$ for $n\toinf$. Then there exists a sequence $\folge{c}{n}\subset\R$ such that for all $n$ sufficiently large with $1/d(f_n,f)>N$ the equality $\abs{f_n(x)-c_n-f(x)}\leq d(f_n,f)$ holds for all $x\in[-N,N]\setminus S_n$, where $S_n$ is a set of measure at most $d(f_n,f)$. Since $N$ is arbitrary, this means that $f_{n}-c_{n}$ converges to $f$ locally in Lebesgue measure. Now let $\subfolge{f}{n}{k}$ be some subsequence of $\folge{f}{n}$. Then $f_{n_k}-c_{n_k}$ converges to $f$ locally in Lebesgue measure  as well. Therefore there exists a further subsequence $f_{n_{k_l}}-c_{k_l}$ which converges to $f$ almost everywhere, see Cohn \cite[Prop. 3.1.3]{cohn2013measure} in conjunction with a diagonal sequence argument. But then $\lim_{l\toinf}(f_{n_{k_l}}(x)-f_{n_{k_l}}(y))=\lim_{l\toinf}(f_{n_{k_l}}(x)-c_{k_l})-\lim_{l\toinf}(f_{n_{k_l}}(y)-c_{k_l})=f(x)-f(y)$ for almost all $x,y\in\R$. Hence $f_n\xRa f$.

   Conversely, assume $f_{n}\xRa f$. Then for any subsequence $\subfolge{f}{n}{k}$ there exists a further subsequence $\subsubfolge{f}{n}{k}{l}$ and a set $S$ of Lebesgue measure zero such that $f_{n_{k_l}}(x)-f_{n_{k_l}}(y)\ra f(x)-f(y)$ for all $x\in\R\setminus S$. Let $x_0\in\R\setminus S$ be fixed and let $c_n\coloneqq f_n(x_0)-f(x_0)$ for all $n\in\N$. Then $f_{n_{k_l}}(x)-c_{n_{k_l}}=f_{n_{k_l}}(x)-f_{n_{k_l}}(x_0)+f(x_0)$, which converges to $f(x)$ for all $x\in\R\setminus S$. In particular, $f_{n_{k_l}}-c_{n_{k_l}}$ converges to $f$ locally in Lebesgue measure (see e.g. \cite[Prop. 3.1.2]{cohn2013measure}), hence clearly $d(f_{n_{k_l}},f)\ra 0$. Since each subsequence contains a further subsequence which converges in the metric $d$ to $f$, the whole sequence must converge in $d$ to $f$ as well.\\
   Hence $d(f_n,f)\ra 0$ if and only if $f_n\xRa f$.
   \QED

Almost basic convergence is however not induced by a norm.
\begin{lemma}
    There exists no norm $\norm{\cdot}$ on $\W$ such that convergence in $\norm{\cdot}$ is equivalent to almost basic convergence.
\end{lemma}
\proof Let $\norm{\cdot}$ be a norm on $\W$. Consider the sequence $\folge{f}{n}$ in $\W$ where $$f_n=\frac{\ind_{[0,1/n)}}{\norm{\ind_{[0,1/n)}}}.$$
Then $\norm{f_n}=1$ for all $n\in\N$, but $f_n(x)=0$ for all $x\in\R\setminus[0,1/n]$ and $n\in\N$. Hence $f_n\Ra 0$ as $n\toinf$, since $f_n(x)\ra 0$ as $n\toinf$ for all $x\in\R\setminus\{0\}$.\qed

\begin{lemma}
    The metric space $(\W,d)$ with $d$ defined as in \cref{lemma:micha_metric} is not complete.
\end{lemma}
    
    \prooof
    Consider sets used in the construction of the Smith-Volterra-Cantor set, but with all intervals being half-open. For this, let $S_0=(0,1]$. We construct $S_1$ by removing an interval of length $4\inv$ in the middle of $S_0$, i.e. $S_1=(0,3/8]\cup(5/8,1]$.
    Then $S_1$ consists of two intervals. In the next step, an interval of the length $4^{-2}$ is removed from the middle of each of these intervals, i.e. $S_2=(0,5/32]\cup(7/32,12/32]\cup(20/32,25/32]\cup(27/32,1]$. This construction is continued inductively, i.e. for all $n\in\N$ the set $S_n$ is the union of $2^n$ half-open intervals, and the set $S_{n+1}$ arises from removing a half-open interval of length $4^{-n-1}$ from the middle of each of them.\\
    Since $\lambda(S_n\setminus S_{n+1})=2^n\cdot 4^{-n-1}=2^{-n-2}$ due to $S_n\setminus S_{n+1}$ being $2^n$ intervals of length $4^{-n-1}$, and since $S_{n+1}\subset S_n$ for all $n\in\N$, we see that the set $S\coloneqq\bigcap_{n=0}^\infty S_n$ has Lebesgue measure $1/2$. Moreover, one can see that for all $x\in S$ and $\varepsilon>0$ it holds that $\lambda([x,x+\varepsilon)\setminus S)>0$.\\
    Since for all $n\in\N$ the set $S_n$ consists of finitely many half-open intervals, the indicator function $f_n\coloneqq \ind_{S_n}$ lies in $\W$. 
    Moreover, $f_n$ converges to $g\coloneqq\ind_S$ pointwise for $n\toinf$. Since $\lim_{x\ra-\infty}g(x)=0$ we see that if $f_n\xRa f$ for some $f\in\W$, then necessarily $f=g$ almost everywhere by the definition of almost basic convergence. But for all $x\in S$ and $\varepsilon>0$ we have $g(x)=1$, and $\lambda(g^{-1}(\{0\})\cap[x,x+\varepsilon))>0$. It follows that for a set $B$ of Lebesgue measure $1/2$ also $f(x)=1$ and $\lambda(f^{-1}(\{0\})\cap[x,x+\varepsilon))>0$ for all $x\in B$ and $\varepsilon>0$. But then $f$ has infinite total variation and is not continuous from the right at any $x\in B$, a contradiction to the assumption $f\in\W$.
    However, $\folge{f}{n}$ is a Cauchy sequence with relation to $d$, since $f_n=f_m$ everywhere except for a set of Lebesgue measure $\abs{2^{-n-1}-2^{-m-1}}$ and therefore $d(f_n,f_m)\leq \abs{2^{-n-1}-2^{-m-1}}$. Hence $(\W,d)$ is not complete. \qed

\begin{lemma}
    The metric space $(\W,d)$ with $d$ defined as in \cref{lemma:micha_metric} is separable.
\end{lemma}
    
    \proof
    Let $M=\{f:\R\ra\R: f=\sum_{i=1}^nc_i\ind_{[a_i,b_i)}, n\in\N, a_i,b_i,c_i\in\Q, a_i<b_i\}$. \\Then $M$ is a countable subset of $\W$. Let $f\in\W$. Then it is easily seen that due to the right-continuity of $f$ and density of $\Q$ in $\R$ there exists a sequence $\folge{f}{n}\subset M$ which converges to $f$ pointwise and hence also almost basically. This shows that $M$ is dense in $(\W,d)$ and hence $(\W,d)$ is separable.\/* For example, let $\varphi:\N\ra\Q$ be bijective and define $\varphi_n:\R\ra\Q$ by $$\varphi_n(x)=\begin{cases}
        0 &, x>\max_{i=1,\ldots,n}\varphi(i)\\
        \min\{\varphi(i):\varphi(i)>x, i\in\{1,\ldots,n\}\} &,\text{otherwise}
    \end{cases}$$
    for all $n\in\N$ and $x\in\R$. Then $\varphi_n(x)\searrow x$ for $n\toinf$ and $\lim_{y\searrow x}\varphi_n(y)=\varphi_n(x)$ for all $x\in\R$. For all $n\in\N$, define $f_n:\R\ra\R$ by
    $$f_n(x)=\varphi_n(f(\varphi_n(x))).$$
    Then $f_n\in M$ for all $n\in\N$, and due to the aforementioned properties of $\varphi_n$,  $f_n$ converges to $f$ pointwise for $n\toinf$.*/

\subsection*{Relationship between types of convergence}
The argumentation in the last step of the proof of \cref{lemma:micha_metric} leads to the following quite natural characterization of almost basic convergence.

\begin{theorem}
\label{theo:micha_equiv}
Let $f_n,f\in\W (n\in\N)$. Then the following are equivalent:
\begin{enumerate}[label=(\roman*)]
    \item $f_n\xRa f$ as $n\toinf$.
    \item $f_n$ converges to $f$ locally in Lebesgue measure up to constants, i.e. there exists a sequence $\folge{c}{n}$ of real numbers such that for all $\varepsilon>0, r>0$
    $$\lim_{n\toinf}\lambda\brc{\{x\in[-r,r]:\abs{f_n(x)-c_n-f(x)}>\varepsilon\}}=0.$$
    \item There exists a sequence $\folge{c}{n}$ of real numbers such that for each subsequence $\subfolge{f}{n}{k}$ of $\folge{f}{n}$ there exists a further subsequence $\subsubfolge{f}{n}{k}{l}$ and a Borel set $S$ of Lebesgue measure zero such that
    $$f_{n_{k_l}}(x)-c_{n_{k_l}}\ra f(x),\quad l\toinf$$
    for all $x\in\R\setminus S$.
    
\end{enumerate}
\end{theorem}
\prooof 
Let $f_n\xRa f$ for $n\toinf$. By \cref{lemma:micha_metric}, $d(f_nf)\ra0$ as $n\toinf$, and in the proof of \cref{lemma:micha_metric} it was shown that $f_n$ then converges to $f$ locally in Lebesgue measure up to constants so that i) implies ii). \\
ii) implies iii) by Cohn \cite[Prop. 3.1.3]{cohn2013measure} in conjunction with a diagonal sequence argument, since $f_n-c_n$ converges to $f$ locally in Lebesgue measure.\\
If iii) holds, then $f_{n_{k_l}}(x)-f_{n_{k_l}}(y)=f_{n_{k_l}}(x)-c_{n_{k_l}}-f_{n_{k_l}}(y)+c_{n_{k_l}}\ra f(x)-f(y)$ for all $x,y\in\R\setminus S$, hence $f_n\xRa f$.\\
\QED

\begin{remark}
    If $f\in\V$, then $\lim_{x\rightarrow-\infty}f(x)$ exists in $\R$, and $f-f(-\infty)\in\W$. If $\folge{f}{n}$ is a sequence in $\V$, then $f_n\xr vf$ if and only if $f_n-f_n(-\infty)\xr vf-f(-\infty)$. Hence \cref{theo:micha_equiv} continues to hold word by word for $f_n,f\in\V$.
\end{remark}

 The equivalence of i) and iii) in \cref{theo:micha_equiv} has a counterpart for basic convergence:
\begin{corollary}
\label{theo:micha_1}
    Let $\folge{f}{n},f\in\V$. Then $f_n\Ra f$ if and only if there exist constants $\folge{c}{n}$ in $\R$ such that for each subsequence $\subfolge{f}{n}{k}$ of $\folge{f}{n}$ there exists a further subsequence $\subsubfolge{f}{n}{k}{l}$ and an at most countable set $S$ such that $f_{n_{k_l}}(x)-c_{n_{k_l}}\ra f(x)$ for $l\toinf$ for all $x\in\R\setminus S$.\\
\end{corollary}

    \prooof
    If the second condition holds, then $f_{n_{k_l}}(x)-f_{n_{k_l}}(y)=(f_{n_{k_l}}(x)-c_{n_{k_l}})-(f_{n_{k_l}}(y)-c_{n_{k_l}})\ra f(x)-f(y)$ for all $x,y\in\R\setminus S$, hence $f_n\Ra f$ as $n\toinf$.\\
    Conversely, assume $f_n\Ra f$ for $n\toinf$. Assume $f_n\in\W$ for all $n\in\N$ and $f\in\W$. Then by \cref{theo:micha_equiv} there exist constants $\folge{c}{n}$ such that for each subsequence $\subfolge{f}{n}{k}$ of $\folge{f}{n}$ there exists a further subsequence $\subsubfolge{f}{n}{k}{l}$ and a set $S_1$ of Lebesgue measure zero such that $f_{n_{k_l}}(x)-c_{n_{k_l}}\ra f(x)$ for $l\toinf$ for all $x\in\R\setminus S_1$. By the definition of basic convergence, there exists an at most countable set $S_0$ and a further subsequence, which we do not rename, such that $f_{n_{k_l}}(x)-f_{n_{k_l}}(y)\ra f(x)-f(y)$ for $l\toinf$ and all $x,y\in\R\setminus S_0$. By fixing some $y\in\R\setminus (S_0\cup S_1)$, this implies $f_{n_{k_l}}(x)-c_{n_{k_l}}\ra f(x)$ as $n\toinf$ for all $x\in\R\setminus S_0$, since $f_{n_{k_l}}(y)-c_{n_{k_l}}\ra f(y)$ as $n\toinf$.\\
    The statement can then be generalized on $\V$ by considering $f_n-\lim_{x\ra-\infty}f_n(x)$ for all $n\in\N$ for a sequence in $
    \V$.\QED

In the proof of the next Theorem, we will use the following statement multiple times, which we reformulate here for the convenience of the reader. It is found in Bogachev \cite[Thm. 1.4.7]{bogachev2018weak} and uses a kind of convergence similar to basic convergence.
\begin{theorem}
\label{theo:bogachev_basic}
    A sequence $\folge{\mu}{n}$ of finite signed measures on $\R$ converges weakly to a finite signed measure $\mu$ if and only if all of the following conditions hold:
    \begin{enumerate}[label=(\roman*)]
        \item \label{enum:boga_i}$\folge{\mu}{n}$ is uniformly tight, that is, for all $\varepsilon>0$ there exists a compact interval $[a,b]$ such that
        $$\abs{\mu_n}(\R\setminus[a,b])<\varepsilon$$
        for all $n\in\N$.
        \item \label{enum:boga_ii}$\sup\limits_{n\in\N}\norm{\mu_n}<\infty$.
        \item \label{enum:boga_iii}For each subsequence $\left(\dis{\mu_{n_k}}\right)_{k\in\N}$ of the sequence of distribution functions $\left(\dis{\mu_n}\right)_{n\in\N}$ there exists a further subsequence $\left(\dis{\mu_{n_{k_l}}}\right)_{l\in\N}$ and an at most countable set $S$ such that $\dis{\mu_{n_{k_l}}}(x)\ra\dis\mu(x)$ for all $x\in\R\setminus S$.
        \end{enumerate}
\end{theorem}

We provide a similar statement for vague convergence, which puts basic convergence into relation with vague convergence.

\begin{theorem}
\label{theo:micha_helly}
    
    Let $\mu_n, \mu\quad (n\in\N)$ be finite signed measures. Then $\mu_n\xr{v}\mu$ if and only if
    \begin{enumerate}[label=(\roman*)]
            \item \label{enum:micha_i}
            $\mu_n$ converges to $\mu$ either basically or almost basically and
            \item \label{enum:micha_ii} $\folge{\mu}{n}$ is locally uniformly bounded in variation, i.e. for all $r\in\R$ we have
    $$\sup_{n\in\N}\abs{\mu_n}([-r,r])\coloneqq C_r<\infty.$$
    \end{enumerate}  
    
    Moreover, in (ii) it is sufficient if only $(\mu^+_n)_{n\in\N}$ or $(\mu^-_n)_{n\in\N}$ is locally bounded in variation for any decomposition of $\folge{\mu}{n}$ into non-negative measures $\left(\mu^+_n\right)_{n\in\N}$ and $\left(\mu^-_n\right)_{n\in\N}$ such that $\mu_n=\mu_n^+-\mu_n^-$ for all $n\in\N$.
\end{theorem}
\prooof
Let $\dis{\mu}$ and $\dis{\mu_n}$ be the distribution functions of $\mu$ and $\mu_n$ for all $n\in\N$. Assume that $\dis{\mu_n}\xRa\dis{\mu}$ and that $\folge{\mu}{n}$ is locally uniformly bounded in variation. Let $f\in C_c^1(\R)$. Choose $r\in\R$ such that $\supp(f)\subset[-r,r].$ Then 
\begin{equation}
\label{eq:boga_part_conv}
\intr f(t)d\mu_n(t)=-\intl[-\infty]^{\infty} f'(t)\dis{\mu_n}(t)dt    
\end{equation}
and
\begin{equation}
\label{eq:boga_part_conv_2}
\intr f(t)d\mu(t)=-\intl[-\infty]^{\infty} f'(t)\dis{\mu}(t)dt,    
\end{equation}

by partial integration. Let $\folge{c}{n}$ be as in \cref{theo:micha_equiv} (iii). Since $f$ has compact support, $\intl[-\infty]^{\infty}f'(x)dx=0$. Hence, if we define $g_n\coloneqq \dis{\mu_n}-c_n$, then
\begin{equation*}
    \intl[-\infty]^{\infty} f'(t)\dis{\mu_n}(t)dt=\intl[-\infty]^{\infty} f'(t)g_n(t)dt
\end{equation*}
and thus with \cref{eq:boga_part_conv}
\begin{equation}
\label{eq:micha_2}
    \intr f(t)d\mu_n(t)=-\intl[-\infty]^{\infty} f'(t)g_n(t)dt.\end{equation}
Now let $\subfolge{\mu}{n}{k}$ be a subsequence of $\folge{\mu}{n}$. By definition of $\folge{c}{n}$ and \cref{theo:micha_equiv}, there exists a further subsequence $\subsubfolge{\mu}{n}{k}{l}$ such that $g_{n_{k_l}}\ra\dis{\mu}$ pointwise $\lambda$-almost everywhere for $l\toinf$. Thus $f'(t)g_{n_{k_l}}(t)$ converges pointwise to $f'(t)\dis{\mu}(t)$ for $\lambda$-almost all $t$ for $l\toinf$. Moreover, the convergence in some points and uniform boundedness in variation of $\folge{g}{n}$ restricted to $[-r,r]$ implies that $\folge{g}{n}$ is uniformly bounded on $[-r,r]$, i.e. $\sup_{x\in [-r,r],n\in\N}\abs{g_n(x)}<\infty$. Then the sequence of functions $\folge{f'g}{n_{k_l}}$ converges pointwise $\lambda$-almost everywhere to $f'\dis{\mu}$, is uniformly bounded and has support in $[-r,r]$, hence the dominated convergence theorem yields
\begin{equation*}
    \intl[-\infty]^\infty f'(t)g_{n_{k_l}}(t)dt\ra\intl[-\infty]^\infty f'(t)\dis{\mu}(t)dt,\quad l\toinf.
\end{equation*}
Together with \cref{eq:micha_2,eq:boga_part_conv_2} we have shown that each subequence $\subfolge{\mu}{n}{k}$ contains a further subsequence $\subsubfolge{\mu}{n}{k}{l}$ such that 
\begin{equation*}
\label{eq:micha_3}
    \intr fd\mu_{n_{k_l}} =-\intl[-\infty]^\infty f'(t)g_{n_{k_l}}(t)\ra -\intl[-\infty]^\infty f'(t)\dis{\mu}(t) =\intr fd\mu
\end{equation*}
for $l\toinf$.
But then the whole sequence $\left(\intr fd\mu_n\right)_{n\in\N}$ must converge to $\intr fd\mu$ as well for $n\toinf$. Hence $\intr fd\mu_n\ra\intr fd\mu$ for all $f\in C_c^1(\R)$.\\
For $f\in C_c(\R)$ with $\supp f\subset[-r,r]$, we can approximate $f$ uniformly with functions $g$ in $C_c^1(\R)$ with $\supp(g)\subset[-r,r]$. Uniform boundedness of the total variations of $\folge{\mu}{n}$ on $[-r,r]$ and convergence $\intr gd\mu_n\ra\intr gd\mu$ then implies $\intr fd\mu_n\ra\intr fd\mu$. Hence $\mu_n\xr{v}\mu$. Since basic convergence implies almost basic convergence, we have shown that $i)$ and $ii)$ imply vague convergence.\\

To show the converse, let $\mu_n,\mu\quad (n\in\N)$ be finite signed measures, not locally uniformly bounded in variation. Then there exists an $r\in\R$ and a subsequence $\subfolge{\mu}{n}{k}$ such that
$$\abs{\mu_{n_k}}([-r,r])>k$$
for all $k\in\N$. Let $\varphi\in C_c(\R)$ be such that $\norm{\varphi}_\infty = 1$ and $\varphi\equiv1$ on $[-r,r]$ and $\varphi\equiv 0$ on $\R\setminus[-r-1,r+1]$.

Moreover, define the signed measure $\tilde\mu$ by $d\tilde\mu=\varphi d\mu$, i.e. $\tilde\mu$ is the signed measure with density $\varphi$ with respect to $\mu$, and analogously let $d\tilde\mu_n=\varphi d\mu$ for all $n\in\N$ .\\
Since $\varphi\equiv 1$ on $[-r,r]$, $\mu$ and $\tilde\mu$ coincide on subsets of $[-r,r]$. The same is true for $\tilde\mu_n$ and $\mu_n$ for all $n\in\N$. This implies that for all $k\in\N$
$$\norm{\tilde\mu_{n_k}}\geq\abs{\tilde\mu_{n_k}}([-r,r])=\abs{\mu_{n_k}}([-r,r])>k,$$
hence the measures $\tilde\mu_{n_k}$ are not bounded in variation. Since each weakly convergent sequence of signed measures is bounded in variation by \cref{theo:bogachev_basic}, this implies that $\tilde\mu_{n_k}\not\xr{w}\tilde\mu$. Hence there exists a function $f\in C_b(\R)$ such that $\intr fd\tilde \mu_{n_k}\not\ra\intr fd\tilde\mu.$\\
Now, since $\varphi$ is the density of $\tilde\mu$ with respect to $\mu$,
$$\intr fd\tilde\mu=\intr f\cdot\varphi d\mu$$
and analogously for all $n\in\N$
$$\intr fd\tilde\mu_n=\intr f\cdot\varphi d\mu_n.$$
Since both $f$ and $\varphi$ are continuous and $\supp(\varphi)\subset [-r-1,r+1]$, the function $f\cdot\varphi$ is in $C_c(\R)$. But by the choice of $f$
$$ \intr f\cdot\varphi d\mu_n = \intr fd\tilde\mu_n \not\ra \intr fd\tilde\mu =\intr f\cdot\varphi d\mu$$
for $n\toinf$.
Hence $\mu_n\not\xr{v}\mu$ for $n\toinf$. We have shown that any vaguely convergent sequence of finite signed measures is locally uniformly bounded in variation.\\
It remains to show that any vaguely convergent sequence of finite signed measures is basically convergent as well. Let $\mu_n,\mu$ be finite signed measures with $\mu_n\xr{v}\mu$ and let $r>0$ be arbitrary, but fixed. Similar to above, let $\varphi_r\in C_c(\R)$ be such that $\norm{\varphi_r}_\infty = 1$ and that $\varphi_r\equiv1$ on $[-r,r]$ and $\varphi_r\equiv 0$ on $\R\setminus[-r-1,r+1]$.
We define the signed measures $\mu_{n}^r,\mu^r$ by $d\mu_n^r=\varphi_rd\mu_n$ and $d\mu^r=\varphi_rd\mu$.\\
Let $f\in C_b(\R)$ be arbitrary, but fixed. Then $f\cdot\varphi_r\in C_c(\R)$, hence 
$$\intr f\cdot\varphi_rd\mu_n\ra\intr f\cdot\varphi_rd\mu\ntoinf$$
due to the vague convergence of $\mu_n$ to $\mu$, thus similar to above
\begin{align*}
    \intr fd\mu_n^r &=\intr f\cdot\varphi_rd\mu_n \ra\intr f\cdot\varphi_rd\mu=\intr fd\mu^r
\end{align*}
for $n\toinf$.
We see that $\mu_n^r\xr{w}\mu^r$. It follows from \cref{theo:bogachev_basic} that $\mu_n^r\Ra\mu^r$. But since $\mu_n^r$ and $\mu_n$ coincide on $[-r,r]$ for all $n\in\N$ and so do $\mu^r$ and $\mu$, the corresponding distribution functions differ only by constants on $[-r,r]$. Hence the basic convergence $\mu_n^r\Ra\mu^r$ implies that $\mu_n\Ra\mu$ on $[-r,r]$, i.e. for each subsequence $\subfolge{\mu}{n}{k}$ of $\folge{\mu}{n}$ there exists a further subsequence $\subsubfolge{\mu}{n}{k}{l}$ and an at most countable set $S$ such that 
$$\dis{\mu_{n_{k_l}}}(x)-\dis{\mu_{n_{k_l}}}(y)\ra\dis\mu(x)-\dis\mu(y)$$
for all $x,y\in [-r,r]\setminus S$ for $l\toinf$.\\
Through a classical diagonal sequence argument, this basic convergence on compact intervals implies $\mu_n\Ra\mu$ on $\R$. Since basic convergence implies almost basic convergence, we have shown that vague convergence implies i) and ii).\\

Regarding the sufficient condition of local boundedness of $(\mu_n^+)_{n\in\N}$ or $(\mu_n^-)_{n\in\N}$, assume that $\mu_n$ converges basically or almost basically to $\mu$. For all $n\in\N$ let $\mu_n^+,\mu_n^-$ be some nonnegative measures such that $\mu_n=\mu_n^+-\mu_n^-$. Then for any $x,y\in\R$
\begin{equation}
    \label{eq:bound_var}
    \dis{\mu_n}(x)-\dis{\mu_n}(y)=\brc{\dis{\mu_n^+}(x)-\dis{\mu
_n^+}(y)} -\brc{\dis{\mu_n^-}(x)-\dis{\mu_n^-}(y)}.
\end{equation}
Now let $x>y$ be arbitrary, but fixed and assume that $(\mu^+_n)_{n\in\N}$ is locally bounded. Then due to the (almost) basic convergence, for any subsequence $\subfolge{\mu}{n}{k}$ there exists a further subsequence $\subsubfolge{\mu}{n}{k}{l}$ and some $\tilde x>x,\tilde y<y
$ such that $\dis{\mu_{n_{k_l}}}(\tilde x)-\dis{\mu_{n_{k_l}}}(\tilde y)$ converges to $\dis\mu(\tilde x)-\dis\mu(\tilde y)$ for $l\toinf$. Therefore $\dis{\mu_{n_{k_l}}}(\tilde x)-\dis{\mu_{n_{k_l}}}(\tilde y)$ is bounded in $l$. Now since $(\mu^+_n)_{n\in\N}$ is locally bounded, it follows from \cref{eq:bound_var} that $\dis{\mu^-_{n_{k_l}}}(\tilde x)-\dis{\mu^-_{n_{k_l}}}(\tilde y)$ is bounded in $l$ as well.\\
Since $\dis{\mu^-_n}$ is monotonically increasing for all $n\in\N$, it follows that $\dis{\mu^-_{n_{k_l}}}(x)-\dis{\mu^-_{n_{k_l}}}(y)$ is bounded in $l$ as well. Since every subsequence of $\brc{\dis{\mu^-_n}(x)-\dis{\mu^-_n}(y)}_{n\in\N}$ contains a subsequence that is bounded, the whole sequence has to be bounded. Hence we see that the local uniform boundedness of $(\mu^+_n)_{n\in\N}$ implies the local uniform boundedness of $(\mu^-_n)_{n\in\N}$, which clearly implies that $\folge{\mu}{n}$ is locally uniformly bounded in variation.\\
The same way, it can also be shown that local uniform boundedness in variation of the negative parts implies local uniform boundedness in variation of the positive parts and hence of the whole sequence.

\QED

The following is now immediate from \cref{theo:micha_helly}:
\begin{corollary}
    Let $\mu_n, \mu\quad (n\in\N)$ be finite signed measures. Assume that $\folge{\mu}{n}$ is locally bounded in variation. Then the following are equivalent:
    \begin{enumerate}
        \item $\mu_n\xr{v}{\mu}$
        \item $\mu_n$ converges basically to $\mu$.
        \item $\mu_n$ converges almost basically to $\mu$.
    \end{enumerate}
\end{corollary}
    
Using \cref{theo:micha_equiv} we can reformulate \cref{theo:micha_helly} as follows:

\begin{corollary}
    Let $\mu_n,\mu\quad(n\in\N)$ be finite signed measures and $\dis{\mu_n}$ and $\dis{\mu}$ their corresponding distribution functions. Then $\mu_n\xr{v}\mu$ if and only if $\folge{\mu}{n}$ is locally uniformly bounded in variation and there exists a sequence $\folge{c}{n}$ in $\R$ such that $\dis{\mu_n}+c_n$ converges to $\dis\mu$ locally in Lebesgue measure.
\end{corollary}

   The connection between vague convergence of signed measures and convergence of their distribution functions has already been studied before in Herdegen et al. \cite{herdegen2022vague}. In their work, they introduced the notion of a sequence of signed measures $\mu_n,n\in\N$ \emph{having no mass} at a point $x\in\R$, if for any $\varepsilon>0$ there exists an open neighbourhood $N_{x,\varepsilon}$ of $x$ such that $$\limsup_{n\toinf}\abs{\mu_n}(N_{x,\varepsilon})\leq\varepsilon.$$
    They then proved:
    
    \begin{theorem}\label{theo:herdegen}{\cite[Thm. 3.8]{herdegen2022vague}}\\
        Let $a\in\R$ and $\mu,\mu_n$ be signed measures ($n\in\N$).
\begin{enumerate}[label=(\alph*)]
    \item If $\dis{\mu_n}(x)-\dis{\mu_n}(a)\ra\dis\mu(x)-\dis\mu(a)$ at all continuity points $x$ of $\dis\mu$ and $\folge{\mu}{n}$ is locally uniformly bounded in variation, then $\mu_n\xr{v}\mu$.
    \item If $\mu_n\xr{v}\mu$, $a$ is a continuity point of $\dis\mu$, and $\folge{\mu}{n}$ has no mass at the continuity points of $\mu$, then $\dis{\mu_n}(x)-\dis{\mu_n}(a)\ra\dis\mu(x)-\dis\mu(a)$ for all continuity points $x$ of $\dis\mu$.
\end{enumerate}
    \end{theorem}

    In contrast to \cref{theo:herdegen}, which considers pointwise convergence of distribution functions (up to constants), our result considers basic convergence, almost basic convergence or local convergence in Lebesgue measure up to constants of the sequence of distribution functions. Thereby we loose the requirement of the sequence having no mass and reach a proper if-and-only-if-statement. In comparison to \cref{theo:micha_helly}, part (a) in \cref{theo:herdegen} has the stronger assumption of pointwise convergence at the continuity points of $\dis\mu$ instead of local convergence in Lebesgue measure, resp. almost basic convergence, and achieves vague convergence as well. Part (b), on the other hand, achieves pointwise convergence instead of convergence in Lebesgue measure, but needs the additional assumption of the sequence having no mass at the continuity points of $\dis\mu$.
     In the following, we will see an example for a sequence of signed measures, which violates the condition of the sequence having no mass everywhere, but converges vaguely.

\begin{example}
\label{ex:mass_everywhere}
    The sequence of signed measures which we considered in \cref{ex:no_mass_vague} violates the condition of having no mass at $x$ for all $x\in[0,1)$. By a slight modification, we get a sequence which even violates this condition for all $x\in\R$.
    For all $n\in\N$, let $k_n\in\N_0$ be such that $2^{k_n}\leq n<2^{k_n+1}$. Define the sequence $\folge{\mu}{n}$ of signed measures  by
    
    $$\mu_n=\brc{\delta_{a_n}-\delta_{b_n}},\quad \text{ where } a_n=k_n{\frac{n-2^{k_n}}{2^{k_n}}}\text{ and } b_n=k_n\frac{n+1-2^{k_n}}{2^{k_n}}.$$
    
    This sequence is similar to the sequence presented in \cref{ex:no_mass_vague}. Its distribution functions do not converge pointwise at any point $x\in\R$, since one can easily see that the sequence $(\dis{\mu_n}(x))_{n\in\N}$ accepts both values $0$ and $1$ infinitely many times for all $x\in\R$. Moreover, for any $x\in\R$ the condition of $\folge{\mu}{n}$ having no mass at $x$ is violated, as $\limsup_{n\toinf}\abs{\mu_n}((x,y))=2$ for all $x,y\in\R, x\leq y$. \\ However, the sequence converges to $0$ vaguely (but not weakly), basically and in Lebesgue measure. 
\end{example}

\begin{remark}
    \begin{enumerate}[label=(\alph*)]
        \item The conditions i) and ii) in \cref{theo:micha_helly} are independent from each other. In \cref{ex:bas_unb_komp} above we found a sequence of measures which converges basically, but is not locally bounded in variation and hence does not converge vaguely.\\
        Conversely, the sequence 
        $$\mu_n=\begin{cases}
            \delta_0,&n \text{ is even}\\
            \delta_1, &n \text{ is odd}
        \end{cases}$$
        is uniformly bounded in variation, but does not converge basically, since there exist subsequences with different pointwise limits.
        \item As seen in \cref{ex:bas_unb}, uniform boundedness on all compact intervals does not imply uniform boundedness on $\R$. We will see below that using the stronger assumption of uniform boundedness on $\R$ in \cref{theo:micha_helly} we get \vag convergence instead of vague convergence.
        \item For non-negative measures, the negative parts are trivially uniformly bounded in variation since they are zero. Hence, for them basic convergence and almost basic convergence are equivalent to vague convergence.
        \item Since vague convergence, local convergence in Lebesgue measure up to constants and local boundedness in variation are local conditions, one can easily generalize \cref{theo:micha_helly} also for sigma-finite signed measures or even more general for suitable mappings from $\bigcup_{k\in\N}B([-k,k])$ to $\R$, using \emph{distribution functions centered at $a$} for some $a\in\R$ instead of classical distribution functions.
    \end{enumerate}
\end{remark}

A characterization similar to \cref{theo:micha_helly} holds for \vag convergence. 
\begin{corollary}
\label{lemma:micha_c0_cc}
  Let $\folge{\mu}{n},\mu$ be finite signed measures. Then $\mu_n\cv\mu$ if and only if ${\mu_n}$ converges to $\mu$ either basically or almost basically and $\folge{\mu}{n}$ is uniformly bounded in variation, i.e. $\sup_{n\in\N}\abs{\mu_n}(\R)<\infty$. Moreover, the \\
  \proof
  If $\folge{\mu}{n}$ is uniformly bounded in variation and almost basically convergent to $\mu$ then $\intr fd\mu_n\ra\intr fd\mu$ for $n\toinf$ for all $f\in C_0(\R)$, since we can approximate functions in $C_0(\R)$ uniformly with functions in $C_c(\R)$ and $\mu_n\xr{v}\mu$ by \cref{theo:micha_helly}.\\
  Conversely, since \vag convergence implies vague convergence, $\mu_n\cv\mu$ also implies $\dis{\mu_n}\Ra\dis\mu$ as $n\toinf$ by \cref{theo:micha_helly}. \\
  Let $\mu_n\xr{v}\mu$. Recall that the space of finite signed measures is a Banach space with relation to the total variation norm, and that $C_0(\R)$ is a Banach space with relation to the supremum norm. We can interpret a finite signed measure $\mu$ as an element of the dual space of $C_0(\R)$ by identifying it with the mapping $f\mapsto\intr fd\mu$. Then the operator norm of $\mu$ is equivalent to the total variation of $\mu$. \\
  If now $\mu_n$ converges vaguely to $\mu$, then $\intr fd\mu_n$ converges to $\intr fd\mu$ for all $f\in C_0(\R)$ and hence $\{\mu_n\}$ as a set of linear operators is pointwise uniformly bounded. By the uniform boundedness principle it then is also bounded in operator norm, i.e. in total variation.
  \qed
\end{corollary} 
Observe that Ambrosio et al.\cite[Def 1.58]{ambrosio2000functions} also provides a relation between \vag convergence and vague convergence together with uniform boundedness in variation.

\begin{remark}
    In \cref{lemma:micha_c0_cc}, it is not enough to have uniform boundednes in variation of either the positive part or the negative parts of the sequence of signed measures. The sequence considered in \cref{ex:bas_unb_komp} is basically convergent and not uniformly bounded in variation, but the corresponding sequence of negative parts of measures is uniformly bounded in variation.
\end{remark}

\cref{theo:micha_helly} and \cref{lemma:micha_c0_cc} also yield:

\begin{corollary}
\label{cor:vague_nonmet}
Vague and \vag convergence are not metrizable on the space of all finite signed measures on $\R$.
\end{corollary}
\prooof
The same argument as for \cref{theo:weak_notmet} can also be applied to vague and \vag convergence, using \cref{theo:micha_helly} and \cref{lemma:micha_c0_cc}, since the sequence $\folge{\mu}{m}$ constructed in \cref{theo:weak_notmet} is not even locally uniformly bounded in variation.

 Regarding metrizability, we can summarize: Loose (resp. vague) convergence is (on the space of finite signed measures) not metrizable. It is however equivalent to (local) uniform boundedness in variation and either basic convergence, which is not metrizable neither (\cref{lemma:basic_not_met}), or almost basic convergence, which is metrizable (\cref{lemma:micha_metric}).

\bibliographystyle{abbrv} 
\bibliography{bibliography}
\end{document}